\documentclass[11pt]{article}

\newcommand{\nc}{\newcommand}
\usepackage[enableskew]{youngtab}
\def\KeyWord#1{$\backslash$\IfColor{$\!\!$\textRed{#1}\textBlack}{#1}$\!\!$}

\usepackage{amsfonts}
\usepackage{amssymb}
\usepackage{amsmath}
\usepackage{amsthm}
\usepackage{enumerate}
\usepackage{hyperref}

\usepackage{etex}
\usepackage{amsgen}
\usepackage{amsopn}
\usepackage{verbatim}
\usepackage{xspace}
\usepackage{multicol}
\usepackage{eepic}
\usepackage{url}
\usepackage{upref}
\usepackage{pgf}
\usepackage{tikz}
\usetikzlibrary{patterns}
\usepackage{pgffor}
\usepackage{pgfcalendar}
\usepackage{pgfpages}
\usepackage{shuffle,yfonts}
\usepackage{mathtools}
\DeclareFontFamily{U}{shuffle}{}
\DeclareFontShape{U}{shuffle}{m}{n}{ <-8>shuffle7 <8->shuffle10}{}

\theoremstyle{plain}
\newtheorem{thm}{Theorem}[section]

\newtheorem{prop}[thm]{Proposition}
\newtheorem{lem-defn}[thm]{Lemma-Definition}

\newtheorem{conj}[thm]{Conjecture}

\theoremstyle{definition}

\newtheorem{rem}[thm]{Remark}

\nc{\db}{{\mathbb D}}
\nc{\AMZ}{{\mathsf{AMZ}}}
\nc{\AMT}{{\mathsf{AMT}}}
\nc{\FMZ}{{\mathsf{FMZV}}}
\nc{\MZV}{{\mathsf{MZV}}}
\nc{\MTZV}{{\mathsf{MTZV}}}
\nc{\gF}{{\varPhi}}
\nc{\gL}{{\Lambda}}
\nc{\ot}{\otimes}
\nc{\fA}{{\mathfrak A}}
\nc{\hfA}{{\widehat{\mathfrak A}}}
\nc{\F}{{\mathbb F}}
\nc{\Z}{{\mathbb Z}}
\nc{\R}{{\mathbb R}}
\nc{\N}{{\mathbb N}}
\nc{\Q}{{\mathbb Q}}
\nc{\CC}{{\mathbb C}}

\nc{\gc}{{\gamma}}
\nc{\gG}{{\Gamma}}
\nc{\om}{{\omega}}
\nc{\vep}{{\varepsilon}}
\nc{\ga}{{\alpha}}
\nc{\gl}{{\lambda}}
\nc{\gb}{{\beta}}
\nc{\gd}{{\delta}}
\nc{\gf}{{\varphi}}
\nc{\gs}{{\sigma}}
\nc{\gt}{{\tau}}
\nc{\gk}{{\kappa}}

\nc{\bgl}{{\boldsymbol{\gl}}}

\nc{\tgz}{{\tilde{\zeta}}}
\nc{\tk}{{\tilde{k}}}
\nc{\lra}{\longrightarrow}
\nc{\lmaps}{\longmapsto}
\nc{\ol}{\overline}
\nc{\gD}{{\Delta}}

\nc{\bfa}{{\bf a}}
\nc{\bfs}{{\bf s}}
\nc{\bfp}{{\bf p}}
\nc{\bfq}{{\bf q}}
\nc{\bfk}{{\bf k}}

\nc{\calA}{{\mathcal A}}
\nc{\calH}{{\mathcal H}}
\nc{\calP}{{\mathcal P}}

\nc{\sha}{\shuffle}
\nc{\Sy}{{\mathcal S}}
\nc{\inv}{{\rm inv}}
\nc{\Rac}{{\mathcal R}}
\nc{\Qxy}{\Q\langle \x,\y\rangle}
\nc{\revs}{\ola}
\nc{\res}{{\rm{res}}}
\nc{\eps}{{\epsilon}}

\nc{\bfxi}{{\boldsymbol \xi}}

\date{}

\begin{document}

\title {\Large\bf
Mordell--Tornheim Zeta Values, Their Alternating Version, \\ and
Their Finite Analogs\footnote{Email: crystal.wang.19@bishops.com, zhaoj@ihes.fr}}
\author{\sc{Crystal Wang and Jianqiang Zhao}\\
\ \\
Department of Mathematics,
The Bishop's School,
La Jolla, CA 92037}
\maketitle

\allowdisplaybreaks

\begin{abstract}
The purpose of this paper is two-fold. First, we consider the classical Mordell--Tornheim zeta values and their alternating
version. It is well-known that these values can be expressed as rational linear combinations of multiple zeta values (MZVs) and the alternating MZVs, respectively. We show that, however, the spaces generated by these values over the rational numbers are in general much smaller than the MZV space and the alternating MZV space, respectively, which disproves a conjecture of Bachmann, Takeyama and Tasaka. Second,
we study supercongruences of some finite sums of multiple integer variables.
This kind of congruences is a variation of the so called finite multiple zeta values when the moduli are
primes instead of prime powers. In general, these objects can be transformed to finite analogs of
the Mordell--Tornheim sums which can be reduced to multiple harmonic sums. This approach not only
simplifies the proof of a few previous results but also generalizes some of them. At the end of the paper, we also
provide a conjecture supported by strong numerical evidence.
\end{abstract}



\medskip

\section{Introduction}

Let $\N$ and $\N_0$ be the set of positive integers and nonnegative integers, respectively.
The classical Mordell--Tornheim zeta values (MTZVs)
are defined as follows. Let $k\ge 2$ be a positive integer.
For all $s_1,\dots,s_{k+1}\in\N$
\begin{equation}\label{equ:MTfuncDef}
   T(s_1,\dots,s_k;s_{k+1}):= \sum_{m_1=1}^\infty \cdots \sum_{m_k=1}^\infty
   \frac{1}{m_1^{s_1}\cdots m_k^{s_k}(m_1+\cdots+m_k)^{s_{k+1}}}.
\end{equation}
Note that in the literature this function is also denoted
by $\zeta_{\rm{MT},k}(s_1,\dots,s_k;s_{k+1})$.
They were first investigated by Tornheim~\cite{Torn} in the case $k=2$,
and later by Mordell~\cite{Mord} and Hoffman~\cite{Hoff} with
$s_1=\cdots=s_k=1$. On the other hand, by \cite[Lemma 3.1]{BradleyZh2010}
every such value can be expressed
as a $\Q$-linear combination of multiple zeta values (MZVs)
which are defined by
\begin{equation*}
\zeta(s_1,\dots,s_d):=\sum_{0<k_1<\cdots<k_d} \frac{1}{k_1^{s_1}\dots k_d^{s_d}}
\end{equation*}
for all $s_1,\dots,s_{d-1}\ge 1, s_d\ge 2$.

After the seminal work of Hoffman \cite{Hoff} and
Zagier \cite{Zag} much more results concerning MZVs have been found
(see the book by the second author \cite{Zhao2015a} and Hoffman's webpage \cite{Hweb}).
Therefore, we can derive a lot of relations between MTZVs. A natural question now arises:
can every MZV be expressed as a $\Q$-linear combination of MTZVs? We will give a negative
answer in section 1.

Similar to the study of MZVs, we may add some alternating signs to MTZVs and call them
alternating MTZVs. Or even more generally, we may consider the multiple variable function
\begin{equation}\label{equ:MTfunc}
MT\left({{s_1,\dots,s_k;s_{k+1}}\atop{z_1,\dots,z_k;z_{k+1}}} \right):=
    \sum_{m_1=1}^\infty \cdots \sum_{m_k=1}^\infty
   \frac{z_1^{m_1}\dots z_k^{m_k}z_{k+1}^{m_1+\cdots+m_k}}
    {m_1^{s_1}\cdots m_k^{s_k}(m_1+\cdots+m_k)^{s_{k+1}}}.
\end{equation}
For instance, Matsumoto et al.'s Mordell--Tornheim $L$-functions (see \cite{MatsumotoNaTs2008})
and the second author's colored Tornheim's double series (see \cite{Zhao2010d}) are
both special cases of \eqref{equ:MTfunc}. When $z_1,\dotsc,z_{k+1}=\pm 1$ we call these values
\emph{alternating MTZVs} and abbreviate them by putting a bar on top of the arguments
whenever the corresponding $z_j$'s are $-1$. For example,
\begin{equation*}
T(\bar3,2;4)=MT\left({{\,3\, ,2;4}\atop{-1,1;1}} \right):=
\sum_{m_1,m_2=1}^\infty \frac{(-1)^{m_1}} {m_1^3 m_2^2 (m_1+m_2)^4}.
\end{equation*}

For any $n,d\in\N$ and $\bfs=(s_1,\dots,s_d)\in\N^d$, we define
the \emph{multiple harmonic sums} (MHSs) and their $p$-restricted version for primes $p$ by
\begin{align*}
H_n(\bfs):= \sum_{0<k_1<\cdots<k_d<n} \frac{1}{k_1^{s_1}\dots k_d^{s_d}},\quad
H_n^{(p)}(\bfs):= \sum_{\substack{0<k_1<\cdots<k_d<n\\ p\nmid k_1,\dots, p\nmid k_d}} \frac{1}{k_1^{s_1}\dots k_d^{s_d}}.
\end{align*}
Here, $d$ is called the \emph{depth}, and $|\bfs|:=s_1+\dots+s_d$ the \emph{weight} of the MHS.
For example, $H_{n+1}(1)$ is often called the $n$th harmonic number.
In general, as $n\to \infty$, we see that $H_n(\bfs)\to \zeta(\bfs)$ which are
the multiple zeta values (MZVs) when $s_d>1$.

More than fifteen years ago, the second author \cite{Zhao2007b} discovered the curious congruence
\begin{equation}\label{equ:BaseCongruence}
  \sum_{\substack{i+j+k=p\\ i,j,k>0}} \frac1{ijk} \equiv -2 B_{p-3} \pmod{p}
\end{equation}
for all primes $p\ge 3$. Since then, several different types of generalizations
have been found. For example, see the references
\cite{ChenZ2017,CHQWZ,MTWZ,ShenCai2012b,XiaCa2010,ZhouCa2007}.
In this paper, we will concentrate on congruences of the following type of sums.
Let $\calP_p$ be the set of positive integers not divisible by $p$. For
any positive integers $r$, $d$, $s_1,\dots,s_d$ and any prime $p$, we define
\begin{equation*}
Z_{p^r} (s_1,\dots,s_d):=\sum_{\substack{k_1+k_2+\dots+k_d=p^r\\ k_1,\dots,k_d\in \calP_p }}
    \frac{1}{k_1^{s_1} k_2^{s_2}\dots k_d^{s_d}}.
\end{equation*}
We will first decompose the above sums into finite Mordell--Tornheim sums (see \eqref{equ:Zdecomp2MT}) which in turn can
be studied using the theory of multiple harmonic sum congruences.

For example, Yang and Cai generalize \eqref{equ:BaseCongruence} in \cite{CaiYa2015} as follows.
For $\ga,\gb,\gc\in\N$, if $w=\ga+\gb+\gc$ is odd and prime $p>w$, then we have
\begin{equation}\label{equ:YCCongruence}
Z_{p^r} (\ga,\gb,\gc)\equiv  p^{r-1} Z_p(\ga,\gb,\gc) \pmod{p^r},
\end{equation}
where
\begin{align*}
Z_p(\ga,\gb,\gc)\equiv &\, \left(\sum_{j=1}^{\max\{\ga,\gb\}} \frac{(-1)^{\ga+\gb-j}}{\ga+\gb+\gc}
\left(\binom{\ga+\gb-j-1}{\ga-1}+\binom{\ga+\gb-j-1}{\gb-1}\right) \binom{\ga+\gb+\gc}{j}\right.  \\
  &\,\left.   +2(-1)^\gc\binom{\ga+\gb}{\ga}\gd_{p-1,\ga+\gb+\gc} \right) B_{p-\ga-\gb-\gc}  \pmod{p}.
\end{align*}
In section \ref{sec:genYC} we will extend \eqref{equ:YCCongruence} to the following:
if $\ga+\gb+\gc$ is even, then for all $r\ge 1$
\begin{equation}\label{equ:YCCongruence2}
Z_{p^r} (\ga,\gb,\gc)\equiv Z_p(\ga,\gb,\gc) p^{2r-2} \pmod{p^{2r}},
\end{equation}
We also determine the value $Z_{p} (\ga,\gb,\gc,\gl)$ when the weight is odd in Theorem~\ref{thm:4vars}.

At the end of the paper, we will present a conjecture related to some families of finite Mordell--Tornheim sums.

\section{Classical (alternating) Mordell--Tornheim zeta values}
It is well-known that every MTZV can be expressed as a $\Q$-linear combination of MZVs (see \cite{BradleyZh2010}).
However, it turns out that the space generated by MZVs is much larger so that MTZVs do not generate whole MZV space over $\Q$.

\begin{thm}\label{thm:MTdim}
Let $\MZV_w$ and $\MTZV_w$ be the $\Q$-vector spaces generated by MZVs and MTZVs of weight $w\ge 3$, respectively.
Then $\MTZV_w=\MZV_w$ for all $3\le w\le 14$. Further, let $P_w$ be the Padovan numbers defined by
$P_2=P_3=P_4=1$ and $P_w=P_{w-2}+P_{w-3}$ for all $w\ge 5$.
Then $\dim_\Q\MTZV_{15}<P_{15}=28$ and for all $w\ge 40$,
$$\dim_\Q\MTZV_w< P_w.$$
\end{thm}
\begin{proof}
Let $p(n)$ be the partition function of any positive integer $n$. A celebrated theorem of
Hardy and Ramanujan \cite{HardyR1918} gives the asymptotic formula (see also \cite[p. 70, (5.1.2)]{And})
\begin{equation*}
    p(n)\sim \frac{1}{4n\sqrt{3}}\exp(\pi \sqrt{2n/3} ) \qquad \text{as } n\to \infty.
\end{equation*}
For any fixed weight $w>2$, let $N_w$ be the number of MZTVs of weight $n$. Then clearly
$$N_w=\sum_{j=2}^{w-1} (p(j)-1)$$
and
\begin{equation*}
    \log(N_w)=O(\sqrt{w}).
\end{equation*}
On the other hand, it is well-known by Zagier's conjecture that $d_w=\dim_\Q \MZV_w$ form the Padovan sequence as given
in the theorem. Further, Brown \cite{Brown2012} showed that $d_w\le P_w$. Thus it is not hard to see that
\begin{equation*}
    \log(d_w)=O(w).
\end{equation*}
Consequently the space generated by MTVs of fixed weight $w$ should be much
smaller than $\MZV_w$ for all sufficiently large $w$.
It turns out that for all $w\ge 40$, we can use the more accurate bound of $N_w$ by partition functions
to obtain that $N_w<P_w$ with the computer-aided computation. Hence
$$\dim_\Q\MTZV_w<P_w$$
for all $w\ge 40$. By straightforward computation with the aid of MAPLE one can
see that for all weight $3\le w\le 14$, the two $\Q$-spaces are the same.
But for weight $w=15$, one already sees that $\MTZV_{15}\le 27<P_{15}=28$.
\end{proof}

\begin{rem}
We want to remark that our Theorem~\ref{thm:MTdim} clearly disproves a conjecture of Bachmann, Takeyama and Tasaka (see
\cite[Conjecture~2.4]{BachmannTaTa2020}).
\end{rem}

Similarly, every alternating MTZV can be expressed as a $\Q$-linear combination of alternating MZVs.
For any $w\ge 3$, let $\AMT_w$ and $\AMZ_w$ be the $\Q$-vector spaces generated by alternating MTZVs
and alternating MZVs of weight $w$, respectively. For example,
\begin{alignat*}{3}
&T(1,1;1)=2\zeta(3), \qquad && T(\bar1,1;1)=-\frac58\zeta(3), \quad && T(\bar1,\bar1;1)=\frac14\zeta(3),\\
&T(1,1,1;1)=\frac{\pi^4}{15},\quad &&
T(\bar1,\bar1,1;1)=\frac{\pi^4}{240},\quad &&
T(\bar1,1,1;1)=-\frac{\pi^4}{72}-\zeta(1,\bar3),\\
&T(1,2;1)=\frac{\pi^4}{72},\quad &&
T(\bar1,\bar2;1)=\frac{\pi^4}{288},\quad &&
T(\{\bar1\}_3;1)=-\frac{\pi^4}{240}+3\zeta(1,\bar3),\\
&T(1,1;2)=\frac{\pi^4}{180},\quad &&
T(\bar1,\bar1;2)=2\zeta(1,\bar3),\quad &&
T(\bar1,2;1)=-\frac{\pi^4}{240}-\zeta(1,\bar3),\\
& \ && T(1,\bar2;1)=-\frac{7\pi^4}{720}+\zeta(1,\bar3),\quad  &&
T(\bar1,1;2)=-\frac{\pi^4}{480}-\zeta(1,\bar3).
\end{alignat*}
Therefore $\AMT_3=\langle \zeta(3)\rangle_\Q$ and $\AMT_4=\langle \zeta(4), \zeta(1,\bar3)\rangle_\Q$. As an alternating analog to Theorem~\ref{thm:MTdim} we have the following result.
\begin{thm}\label{thm:AMTdim}
Let $F_w$ be the Fibonacci numbers defined by $F_0=F_1=1$ and $F_w=F_{w-1}+F_{w-2}$ for all $w\ge 2$.
Then for all $3\le w\le 12$ and $w\ge 34$, we have
$$\dim_\Q\AMT_w< F_w $$
\end{thm}
\begin{proof} When $3\le w\le 12$ we computed the set of generators of $\AMT_w$ (see Remark~\ref{rem:AMTdim}).
Let $A_w$ be the number of alternating MTZVs of weight $w$. For each MTZV of weight $w$, we first determine
how many ways to put some alternating signs. Suppose we have such an MTZV
\begin{equation*}
T(\{s_1\}_{j_1},\dotsc,\{s_r\}_{j_r};w-i), \quad s_1<\dotsm<s_r,
\end{equation*}
where for each string $\bfs$ we denote the string obtained by repeating $\bfs$ exactly $j$ times by $\{\bfs\}_j$.
Now for each $\{s_\ell\}_{j_\ell}$ ($1\le \ell\le r$) there are $j_\ell+1$ ways to put alternating signs because
of the symmetry. Thus the number of ways to put some alternating signs on this MTZV is
\begin{equation}\label{equ:numbSign}
    \prod_{\ell=1}^r (j_\ell+1) \qquad
\text{subject to the condition} \qquad
     \sum_{\ell=1}^r  j_\ell s_\ell=i.
\end{equation}
Then it is not hard to see that for fixed $i$ the maximal value of \eqref{equ:numbSign} is achieved when
$i=m(m+1)/2$ is a triangular number and the MTZV is $T(1,2,\dotsc,m;w-i)$. In this case,
the value of \eqref{equ:numbSign} is $2^m$ where $m=(\sqrt{8i+1}-1)/2$. Moreover, the subspace $\MTZV_{34}$ has dimension bounded by the Padovan number $P_{34}$. Thus
$$\dim_\Q\AMT_w\le A_w-N_w+P_w\le \sum_{i=2}^{w-1} \left\lfloor \frac {\sqrt{8i+1}-1}2 \right \rfloor(p(i)-1)-N_w+P_w <F_w$$
for all $w\ge 34$ by computer computation.
\end{proof}

In general, we have the following conjecture.

\begin{conj}\label{conj:AMTbasis}
For every $w\ge 3$, $\AMT_w$ can be generated by the following set of elements
\begin{equation}\label{equ:AMTBasis}
C_w:=\left\{ \prod \zeta(\bfk;1,\dots,1,-1) (2\pi i)^{2n}:2n+\sum_\bfk \lambda(\bfk)|\bfk|=w, n\ge 0\right\},
\end{equation}
where the product runs through all possible Lyndon words $\bfk\ne (1)$
on odd numbers (with $1<3<5<\cdots$) with multiplicity $\lambda(\bfk)$ so that $2n+\sum_{\bfk} \lambda(\bfk)|\bfk|=w$.
\end{conj}

\begin{prop}
If Conjecture \ref{conj:AMTbasis} holds then
$$\dim_\Q\AMT_w\le F_{w-2} \quad \forall w\ge 3.$$
\end{prop}
\begin{proof}
Deligne proved that \cite[Thm. 7.2]{Deligne2010} for all $w\ge 1$ the $\Q$-vector space $\AMZ_w$
of alternating MZVs can be generated by
\begin{equation}\label{equ:DelBasis}
B_w:=\left\{ \prod \zeta(\bfk;1,\dots,1,-1) (2\pi i)^{2n}:
2n+\sum_\bfk \lambda(\bfk)|\bfk|=w, n\ge 0\right\},
\end{equation}
where the product runs through all possible Lyndon words $\bfk$
on odd numbers (with $1<3<5<\cdots$) with multiplicity $\lambda(\bfk)$ so that $2n+\sum_{\bfk} \lambda(\bfk)|\bfk|=w$.
Note that the ordering of indices in the definition of Euler sums is opposite in loc. sit.
So the definition of Lyndon words here has opposite order, too. Furthermore, if a period conjecture of
Grothendieck~\cite[Conjecture~5.6]{Deligne2010} holds
then $B_w$ is a basis of $\AMZ_w$. In particular, $\sharp B_w=F_w$ is the Fibonacci number.
Hence, if Conjecture \ref{conj:AMTbasis} holds then $\AMT_w$ is generated by $C_w=B_{w}\setminus \zeta(\bar1)B_{w-1}$ and
therefore
$$\dim_\Q\AMT_w\le F_w-F_{w-1}=F_{w-2} \quad \forall w\ge 3,$$
as desired.
\end{proof}

\begin{rem}\label{rem:AMTdim}
Using Maple and the table of values for alternating MZVs provided by \cite{BlumleinBrVe2010}
we have verified Conjecture \ref{conj:AMTbasis} for weight $w\le 12$.
\end{rem}
It turns out if Grothendieck Conjecture \cite[Conjecture~5.6]{Deligne2010} holds then
$$\dim_\Q\AMT_w=F_{w-2} \quad \forall 3\le w\le 10.$$
But already in weight $w=11$,
$$\dim_\Q\AMT_{11}=F_9-1=54.$$
In fact, to find the set of generators for $\AMT_{11}$,
one only needs to modify $B_{11}\setminus \zeta(\bar1)B_{10}$ by replacing
the two elements $\zeta(1,1,\bar3)\zeta(1,1,1,\bar3)$ and $\zeta(1,1,1,3,1,1,\bar3)$ by their linear combination
$2\zeta(1,1,\bar3)\zeta(1,1,1,\bar3)+\zeta(1,1,1,3,1,1,\bar3)$.

\section{Supercongrence related to finite Mordell--Tornheim zeta values}\label{sec:genYC}
Recall that $\calP_p$ is the set of positive integers not divisible by $p$.
For any prime $p$ and positive integer $r$, we call the sum
\begin{equation*}
T_{p^r}(\ga_1,\dots,\ga_m;\gl):=
\sum_{\substack{k_1+\dots+k_m, k_1,\dots,k_m\in\calP_p\\ k_1+\dots+k_m<p^r}}
\frac{1}{k_1^{\ga_1} \cdots k_m^{\ga_m} (k_1+\dots+k_m)^\gl}
\end{equation*}
a \emph{finite Mordell--Tornheim sum}.
We then define the $p$-restricted finite Mordell--Tornheim sums as follows.
For any $m,n,r\in\N$ and $\ga_1,\dots,\ga_m,\gl_1,\dots,\gl_n\in\N_0$, we set
\begin{equation*}
T_{p^r}(\ga_1,\dots,\ga_m;\gl_1,\dots,\gl_n):= \sum_{\substack{k_1+\dots+k_m=u_1<\dots<u_n<p^r \\ k_1,\dots,k_m,u_1,u_2-u_1,\dots,u_n-u_{n-1},u_n\in\calP_p}}
\frac{1}{k_1^{\ga_1} \cdots k_m^{\ga_m} u_1^{\gl_1}\cdots u_n^{\gl_n}}.
\end{equation*}
Here, we call $m+n-1$ the depth and $\ga_1+\dots+\ga_m+\gl_1+\dots+\gl_n$ the weight
of this sum.

By definition we have
\begin{align}
&\, Z_{p^r}(\ga_1,\dots,\ga_{n+1}) \notag \\
=&\, \sum_{\substack{k_1+\dots+k_{n+1}=p^r\\ k_1,\dots,k_{n+1}\in\calP_p}}
\frac{1}{k_1^{\ga_1}\cdots k_{n+1}^{\ga_{n+1}}}
=\sum_{\substack{u=k_1+\dots+k_n<p^r\\ k_1,\dots,k_n,u\in\calP_p}}
\frac{1}{k_1^{\ga_1}\cdots k_n^{\ga_n} (p^r-u)^{\ga_{n+1}}} \notag \\
=&\, (-1)^{\ga_{n+1}} \sum_{\substack{u=k_1+\dots+k_n<p^r\\ k_1,\dots,k_n,u\in\calP_p}}
\frac{1}{k_1^{\ga_1}\cdots k_n^{\ga_n}  u^{\ga_{n+1}}} \left(1-\frac{p^r}{u}\right)^{-\ga_{n+1}} \notag\\
\equiv &\,(-1)^{\ga_{n+1}} \sum_{\substack{u=k_1+\dots+k_n<p^r\\ k_1,\dots,k_n,u\in\calP_p}}
\left(\frac{1}{k_1^{\ga_1}\cdots k_n^{\ga_n}  u^{\ga_{n+1}} }
    +\frac{\ga_{n+1} p^r}{k_1^{\ga_1}\cdots k_n^{\ga_n}  u^{\ga_{n+1}+1}}\right) \pmod{p^{2r}}  \notag\\
 \equiv &\,(-1)^{\ga_{n+1}} \Big(T_{p^r}(\ga_1,\dots,\ga_n;\ga_{n+1})
    +\ga_{n+1} p^r T_{p^r}(\ga_1,\dots,\ga_n;\ga_{n+1}+1) \Big)\pmod{p^{2r}}. \label{equ:Zdecomp2MT}
\end{align}
Therefore, we have decomposed $Z_{p^r}(\ga_1,\dots,\ga_{n+1})$
as a sum of finite Mordell--Tornheim sums.

Define
\begin{equation*}
\calH_{p^r}(s_1,\dots,s_d):=\sum_{\substack{0<u_1<\cdots<u_d<p^r\\ u_1,u_2-u_1,\dots,u_d-u_{d-1},u_d\in\calP_p}}\frac{1}{ u_1^{s_1} \cdots u_d^{s_d}}.
\end{equation*}

\begin{thm} \label{thm:DblReducedMHS}
Let $p$ be a prime, $a,b,r\in\N$ such that $p>w=a+b$. If $w$ is odd then
\begin{equation*}
\calH_{p^r}(a,b)\equiv p^{r-1} \calH_p(a,b) \pmod{p^r}.
\end{equation*}
If $w$ is even then
\begin{equation*}
\calH_{p^r}(a,b)\equiv p^{2r-2} \calH_p(a,b) \pmod{p^{2r}}.
\end{equation*}
\end{thm}
\begin{proof} The case $r=1$ is trivial so we may assume $r\ge 2$.

First we assume the weight is even.
By Euler's theorem, setting $m=\varphi(p^{2r})-a$ and
$n=\varphi(p^{2r})-b$ we get
\begin{alignat*}{4}
\calH_{p^r}(a,b)&\, \equiv \sum_{k<l<p^r; k,l,l-k\in\calP_p} k^m l^n & &\pmod{p^{2r}} \\
&\, \equiv \sum_{k<l<p^r; k,l\in\calP_p} k^m l^n - \sum_{t<p^{r-1},\ k+pt<p^r;k\in\calP_p} k^m (k+pt)^n
& & \pmod{p^{2r}} \\
&\, \equiv \sum_{k<l<p^r} k^m l^n - \sum_{t<p^{r-1}}\sum_{k<p^r-pt} k^m (k+pt)^n
& &\pmod{p^{2r}}
\end{alignat*}
since $\min\{m,n\}>\varphi(p^{2r})-w\ge (p^{2r-1}-1)(p-1)\ge 2r.$ Now
\begin{alignat*}{4}
&  \sum_{k<l<p^r} k^m l^n = \sum_{j=0}^{m} \binom{m+1}{j} \frac{B_j}{m+1} \sum_{l<p^r} l^{m+1-j+n}  \\
&\, = \sum_{j=0}^{m} \binom{m+1}{j} \frac{B_j}{m+1}
\sum_{i=0}^{m+n+1-j}  \binom{m+n+2-j}{i} \frac{B_i  p^{r(m+n+2-j-i)}}{m+n+2-j}  \\
&\, \equiv p^r \sum_{j=0}^{m} \binom{m+1}{j} \frac{B_jB_{m+n+1-j}}{m+1}
+  p^{2r} \sum_{j=0}^{m} \binom{m+1}{j} \frac{ (m+n+1-j) B_j B_{m+n-j}}{2(m+1)}
   & & \pmod{p^{2r}} \\
&\, \equiv -\frac{p^r B_{m+n}}{2}
+  p^{2r} \sum_{j=0}^{m} \binom{m+1}{j} \frac{ (m+n+1-j) B_j B_{m+n-j}}{2(m+1)}
   & &\pmod{p^{2r}}
\end{alignat*}
since $m+n$ is even. Note that in the last sum above, if $(p-1)|j$ or $(p-1)|(m+n-j)$
then $B_j B_{m+n-j}$ may not be $p$-integral, but $pB_j B_{m+n-j}$ must be since
$(p-1)\nmid(m+n)$. On the other hand,
\begin{align*}
&\, \sum_{t<p^{r-1}}\sum_{k<p^r-pt} k^m (k+pt)^n \\
&\, =
\sum_{s=0}^n \binom{n}{s}  \sum_{t<p^{r-1},\ k<p^r-pt} (pt)^{n-s} k^{m+s} \\
&\, \equiv \sum_{s=0}^n \binom{n}{s}
\sum_{j=0}^{m+s} \binom{m+s+1}{j} \frac{B_j}{m+s+1}
\sum_{t<p^{r-1}} (p^r-pt)^{m+s+1-j} (pt)^{n-s}
\pmod{p^{2r}} \\
&\, \equiv \sum_{s=0}^n \binom{n}{s} \sum_{j=0}^{m+s} \binom{m+s+1}{j} \frac{B_j}{m+s+1} \\
&\, \ \hskip2.5cm \sum_{t<p^{r-1}} \big[(m+s+1-j)p^r(-pt)^{m+s-j}+(-pt)^{m+s+1-j}\big] (pt)^{n-s}  \pmod{p^{2r}}\\
&\, \equiv B_{m+n} \sum_{t<p^{r-1}} p^r+\sum_{s=0}^n \sum_{\ j=0;\ j\ne m+n}^{m+s} \sum_{i=0}^{m+n-j} \binom{n}{s} \binom{m+s+1}{j}  \binom{m+n+1-j}{i} \\
&\, \ \hskip2.5cm \cdot \frac{(m+s+1-j)(-1)^{m+s-j}B_jB_i}{(m+s+1)(m+n+1-j)}  p^{r+m+n-j+(r-1)(m+n+1-j-i)}  \\
&\, - p B_{m+n} \sum_{t<p^{r-1}}t +\sum_{s=0}^n \binom{n}{s} \sum_{\ j=0;\ j\ne m+n}^{m+s} \sum_{i=0}^{m+n+1-j} \binom{m+s+1}{j}  \binom{m+n+2-j}{i} \\
&\, \ \hskip2.5cm \cdot \frac{(-1)^{m+s+1-j}B_jB_i}{(m+s+1)(m+n+2-j)}
 p^{m+n+1-j+(r-1)(m+n+2-j-i)}
\pmod{p^{2r}} \\
&\, \equiv B_{m+n} \frac{ p^r(p^{r-1}-1)}2 \pmod{p^{2r}}.
\end{align*}
Combining all the above together, we see that
\begin{align*}
\calH_{p^r}(a,b) \equiv&\,-\frac{p^{2r-1}B_{m+n} } 2+ p^{2r} \sum_{j=0}^{m} \binom{m+1}{j} \frac{ (m+n+1-j) B_j B_{m+n-j}}{2(m+1)} \\
\equiv&\, p^{2r-2}\calH_{p}(a,b)
   \pmod{p^{2r}}
\end{align*}
which follows from the proof of the case with $r=1$ while keeping $m=\varphi(p^{2r})-a$ and
$n=\varphi(p^{2r})-b$.
This completes the proof of the theorem when the weight is even. The proof
of the odd weight case is similar but simpler so we leave it
to the interested reader.
\end{proof}

\begin{thm}\label{thm:extYangCai}
For all $r,\ga,\gb,\gc\in\N$ and primes $p>\ga+\gb+\gc$, if $\ga+\gb+\gc$ is odd then we have
\begin{equation}\label{equ:YCCongruence3}
Z_{p^r} (\ga,\gb,\gc)\equiv Z_p(\ga,\gb,\gc) p^{r-1} \pmod{p^r}.
\end{equation}
Furthermore, if $\ga+\gb+\gc$ is even, then
\begin{equation}\label{equ:YCCongruence4}
Z_{p^r} (\ga,\gb,\gc)\equiv Z_p(\ga,\gb,\gc) p^{2r-2} \pmod{p^{2r}}.
\end{equation}
\end{thm}

\begin{proof}
By a result of Bradley and Zhou,
it can be shown that all Mordell--Tornheim sums can be reduced to the
finite multiple zeta values defined in the introduction.
Indeed, by \cite[Lemma 3.1]{BradleyZh2010}, when $n=2$, we have
\begin{align*}
T_{p^r}(\ga,\gb;\gc)=&\, \sum_{\substack{u=k_1+k_2<p^r\\ k_1,k_2,u\in\calP_p}}
\left(\sum_{a=0}^{\ga-1} \binom{a+\gb-1}{a}
\frac{k_1^a k_2^\gb}{u^{a+\gb} }
+\sum_{b=0}^{\gb-1} \binom{b+\ga-1}{b}
\frac{k_1^\ga k_2^b}{u^{\ga+b} }  \right) \frac{1}{k_1^\ga k_2^\gb u^\gc }\\
=&\, \sum_{a=0}^{\ga-1} \binom{a+\gb-1}{a}
\calH_{p^r}(\ga-a,a+\gb+\gc)
+\sum_{b=0}^{\gb-1} \binom{b+\ga-1}{b}
\calH_{p^r}(\gb-b,b+\ga+\gc)  \\
=&\, \sum_{a=1}^{\ga} \binom{\ga+\gb-a-1}{\ga-a}
\calH_{p^r}(a,w-a)
+\sum_{b=1}^{\gb} \binom{\ga+\gb-b-1}{\gb-b}
\calH_{p^r}(b,w-b),
\end{align*}
where $w=\ga+\gb+\gc$. Thus the theorem follows from the decomposition
formula \eqref{equ:Zdecomp2MT} and Theorem~\ref{thm:DblReducedMHS} immediately.
\end{proof}

When $n=3$, the situation is completely similar although the formulas are more involved.
By straight-forward computation, we get
\begin{align*}
T_{p^r}(\ga,\gb,\gc;\gl)=&\, \sum_{\substack{u=k_1+k_2+k_3<p^r\\ k_1,k_2,k_3,u\in\calP_p}}
\left(\sum_{a=0}^{\ga-1} \sum_{b=0}^{\gb-1} \binom{a+b+\gc-1}{a,b,\gc-1}
\frac{k_1^a k_2^b k_3^\gc}{u^{\gc+a+b} } \right. \\
&\, \ \qquad \ \qquad +\sum_{a=0}^{\ga-1} \sum_{c=0}^{\gc-1} \binom{a+c+\gb-1}{a,c,\gb-1}
\frac{k_1^a k_2^\gb k_3^c}{u^{\gb+a+c} } \\
&\, \left. \ \qquad \ \qquad +\sum_{b=0}^{\gb-1} \sum_{c=0}^{\gc-1} \binom{b+c+\ga-1}{b,c,\ga-1}
\frac{k_1^\ga k_2^b k_3^c}{u^{\ga+b+c} }
\right) \frac{1}{k_1^\ga k_2^\gb k_3^\gc u^\gl}\\
=&\, \sum_{a=0}^{\ga-1} \sum_{b=0}^{\gb-1} \binom{a+b+\gc-1}{a,b,\gc-1}
T_{p^r}(\ga-a,\gb-b,0;\gl+\gc+a+b) \\
+&\,\sum_{a=0}^{\ga-1} \sum_{c=0}^{\gc-1} \binom{a+c+\gb-1}{a,c,\gb-1}
T_{p^r}(\ga-a,\gc-c,0;\gl+\gb+a+c)\\
+&\,\sum_{b=0}^{\gb-1} \sum_{c=0}^{\gc-1} \binom{b+c+\ga-1}{b,c,\ga-1}
T_{p^r}(\gb-b,\gc-c,0;\gl+\ga+b+c)\\
=&\, \sum_{a=1}^{\ga} \sum_{b=1}^{\gb} \binom{n-a-b-1}{\ga-a,\gb-b,\gc-1}
T_{p^r}(a,b,0;w-a-b) \\
+&\,\sum_{a=1}^{\ga} \sum_{c=1}^{\gc} \binom{n-a-c-1}{\ga-a,\gc-c,\gb-1}
T_{p^r}(a,c,0;w-a-c) \\
+&\,\sum_{b=1}^{\gb} \sum_{c=1}^{\gc} \binom{n-b-c-1}{\gb-b,\gc-c,\ga-1}
T_{p^r}(b,c,0;w-b-c)
\end{align*}
where $n=\ga+\gb+\gc$ and $w=\ga+\gb+\gc+\gl$.
Applying \cite[Lemma 3.1]{BradleyZh2010} again (or \cite[Lemma 2.8]{CaiYa2015}) we see
that
\begin{align*}
T_{p^r}(\ga,\gb,0;\gl)=&\,
\sum_{s=0}^{\ga-1} \binom{s+\gb-1}{s} \sum_{\substack{k_1<u_2<u_3<p^r\\ k_1,u_2-k_1,u_3-u_2,u_3\in\calP_p}}
\frac{1}{ k_1^{\ga-s} u_2^{\gb+s} u_3^\gl} \\
+&\, \sum_{t=0}^{\gb-1} \binom{t+\ga-1}{t} \sum_{\substack{k_2<u_2<u_3<p^r\\ k_2,u_2-k_2,u_3-u_2,u_3\in\calP_p}}
\frac{1}{ k_2^{\gb-t} u_2^{\ga+t} u_3^\gl} \\
=&\,
\sum_{s=0}^{\ga-1} \binom{s+\gb-1}{s}  \calH_{p^r}(\ga-s,\gb+s,\gl)
+\sum_{t=0}^{\gb-1} \binom{t+\ga-1}{t} \calH_{p^r}(\gb-t,\ga+t,\gl).
\end{align*}
Then we get
\begin{align*}
&\, (-1)^\gl Z_{p^r}(\ga,\gb,\gc,\gl) \\
\equiv &\,
 \sum_{a=1}^{\ga} \sum_{b=1}^{\gb} \binom{n-a-b-1}{\ga-a,\gb-b,\gc-1} \sum_{s=0}^{a-1} \binom{s+b-1}{s} \\
 & \ \quad \Big(\calH_{p^r}(a-s,b+s,w-a-b)
 +\gl p^r \calH_{p^r}(a-s,b+s,w-a-b+1)\Big) \\
 + &\,\sum_{a=1}^{\ga} \sum_{b=1}^{\gb} \binom{n-a-b-1}{\ga-a,\gb-b,\gc-1} \sum_{t=0}^{b-1} \binom{t+a-1}{t} \\
 & \ \quad \Big(\calH_{p^r}(b-t,a+t,w-a-b)
 +\gl p^r \calH_{p^r}(b-t,a+t,w-a-b+1)\Big) \\
 +&\,\sum_{a=1}^{\ga} \sum_{c=1}^{\gc} \binom{n-a-c-1}{\ga-a,\gc-c,\gb-1} \sum_{s=0}^{a-1} \binom{s+c-1}{s}\\
 & \ \quad \Big(\calH_{p^r}(a-s,c+s,w-a-c)
  +\gl p^r \calH_{p^r}(a-s,c+s,w-a-c+1)\Big) \\
 +&\,\sum_{a=1}^{\ga} \sum_{c=1}^{\gc} \binom{n-a-c-1}{\ga-a,\gb-b,\gb-1} \sum_{t=0}^{c-1} \binom{t+a-1}{t}\\
 & \ \quad \Big(\calH_{p^r}(c-t,a+t,w-a-c)
  +\gl p^r \calH_{p^r}(c-t,a+t,w-a-c+1)\Big) \\
 + &\,\sum_{b=1}^{\gb} \sum_{c=1}^{\gc} \binom{n-b-c-1}{\gb-b,\gc-c,\ga-1}\sum_{s=0}^{b-1} \binom{s+c-1}{s}\\
 & \ \quad \Big(\calH_{p^r}(b-s,c+s,w-b-c)
  +\gl p^r \calH_{p^r}(b-s,c+s,w-b-c+1)\Big) \\
 + &\,\sum_{b=1}^{\gb} \sum_{c=1}^{\gc} \binom{n-b-c-1}{\gb-b,\gc-c,\ga-1}\sum_{t=0}^{c-1} \binom{t+b-1}{t}\\
 & \ \quad \Big(\calH_{p^r}(c-t,b+t,w-b-c)
  +\gl p^r \calH_{p^r}(c-t,b+t,w-b-c+1)\Big)
\pmod{p^{2r}}.
\end{align*}
The above computation quickly yields the following result.

\begin{thm} \label{thm:4vars}
Let $p$ be a prime and $\ga,\gb,\gc,\gl\in\N$ such that $w=\ga+\gb+\gc+\gl$ is odd.
If $p>w+2$ then we have
\begin{align*}
Z_p(\ga,\gb,\gc,\gl)\equiv &\,(-1)^\gl
\left( \sum_{a=1}^{\ga} \sum_{b=1}^{\gb} \binom{n-a-b-1}{\ga-a,\gb-b,\gc-1}
t_p(a,b;w-a-b) \right. \\
+&\,\sum_{a=1}^{\ga} \sum_{c=1}^{\gc} \binom{n-a-c-1}{\ga-a,\gc-c,\gb-1}
t_p(a,c;w-a-c) \\
+&\,\left.\sum_{b=1}^{\gb} \sum_{c=1}^{\gc} \binom{n-b-c-1}{\gb-b,\gc-c,\ga-1}
t_p(b,c;w-b-c)
\right) B_{p-w} \pmod{p}
\end{align*}
where $n=\ga+\gb+\gc$,
\begin{align*}
t_p(\ga,\gb;\gl)=
\sum_{s=0}^{\ga-1} \binom{s+\gb-1}{s}  h_p(\ga-s,\gb+s,\gl)
+\sum_{t=0}^{\gb-1} \binom{t+\ga-1}{t} h_p(\gb-t,\ga+t,\gl)
\end{align*}
and
\begin{equation*}
h_p(\ga,\gb,\gc)=\frac{1}{2n}\left((-1)^\ga\binom{n}{\ga}-(-1)^\gc \binom{n}{\gc}\right).
\end{equation*}
\end{thm}
\begin{proof}
Observe that
\begin{equation*}
\calH_{p}(\ga,\gb,\gc)=H^{(p)}_p(\ga,\gb,\gc)=H_p(\ga,\gb,\gc).
\end{equation*}
Taking $r=1$ in the above computation, we see that the theorem follows from
\cite[Thm. 1.8]{Zhao2008a} or \cite[Thm. 8.5.13]{Zhao2015a} quickly.
\end{proof}

The following conjecture is supported by some extensive numerical evidence.

\begin{conj}
Let $r\in\N$, $p$ be a prime and $\bfs\in\N^d$ such that $p>|\bfs|+1$.
\begin{enumerate}

  \item If $d=4$ and $|\bfs|$ is odd:
\begin{equation}\label{eqn:conjEvenwc}
Z_{p^r}(\bfs)\equiv p^{2r-2} Z_p(\bfs) \pmod{p^{2r-1}}.
\end{equation}

  \item If $d=4$ and $ |\bfs|$ is even:
\begin{equation}\label{eqn:conjEvenwd}
Z_{p^r}(\bfs)\equiv p^{r-1} Z_p(\bfs) \pmod{p^{r}}.
\end{equation}

  \item If $d=5$ and $|\bfs|$ is even:
\begin{equation}\label{eqn:conjEvenwf}
Z_{p^r}(\bfs)\equiv p^{2r-2} Z_p(\bfs) \pmod{p^{2r-1}}.
\end{equation}
\end{enumerate}
In general, if $r\ge 2$ and $d+|\bfs|$ is odd then we have
\begin{equation}\label{eqn:conjEvenwk}
Z_{p^r}(\bfs)\equiv 0 \pmod{p^{2r-2}}.
\end{equation}
If $r\ge 2$ and $d+|\bfs|$ is even then we have
\begin{equation}\label{eqn:conjEvenwl}
Z_{p^r}(\bfs)\equiv 0 \pmod{p^{r-1}}.
\end{equation}
\end{conj}

In general, the powers of moduli in \eqref{eqn:conjEvenwc}--\eqref{eqn:conjEvenwf} cannot be increased. For example,
\begin{align*}
Z_{13^3}(8,1,1,1) \equiv&\,  13^4  Z_{13}(8,1,1,1) \pmod{13^{5}},
\quad\text{but}\\
Z_{13^3}(8,1,1,1) \not \equiv&\,   13^4  Z_{13}(8,1,1,1) \pmod{13^{6}}.
\end{align*}

We further remark that the patterns in  \eqref{eqn:conjEvenwc}--\eqref{eqn:conjEvenwf} do not seems to continue
for larger depths even though \eqref{eqn:conjEvenwk} and \eqref{eqn:conjEvenwl} should hold for all $d$.
This is also consistent with the parity phenomenon such that when the weight and the depth of $Z_{p^r}(\bfs)$
have different parities it can be ``reduced further'', similar to the classical situation for the multiple zeta values.
A detailed description of a conjectural link between the classical version of these values and their ``finite''
analogs can be found in Chapter 8 of \cite{Zhao2015a}.


\end{document}